\documentclass[a4paper,10pt]{article}
\usepackage{pgf,tikz}
\usetikzlibrary{arrows}
\usepackage{geometry}
\usepackage{latexsym}
\usepackage{amsfonts, amsmath, amssymb, amsthm}
\usepackage[utf8]{inputenc}
\usepackage[T1]{fontenc}
\usepackage[english]{babel}
\usepackage{stmaryrd}
\usepackage{subfig}
\usepackage{frcursive}
\usepackage{mathrsfs}
\usepackage{enumitem}
\usepackage{yhmath}
\usepackage{palatino}
\usepackage{bbm}

\newcommand{\N}{\mathbb{N}}
\newcommand{\Z}{\mathbb{Z}}

\newcommand{\R}{\mathbb{R}}
\newcommand{\C}{\mathbb{C}}

\newtheorem{theorem}{Theorem}[section]
\newtheorem{lemma}[theorem]{Lemma}

\newtheorem{cor}[theorem]{Corollary}

\theoremstyle{definition}
\newtheorem{remark}[theorem]{Remark}

\title{Limit laws for random matrix products}
\author{Jordan Emme\thanks{Universit\'{e} Paris-Sud, CNRS, LMO, UMR 8628, 91405 Orsay, France. E-mail: jordan.emme@math.u-psud.fr} \ and Pascal Hubert\thanks{Aix-Marseille Universit\'{e}, CNRS, Centrale Marseille, I2M, UMR 7373, 13453 Marseille, France.  E-mail: pascal.hubert@univ-amu.fr}}
\date{}

\begin{document}

\maketitle

\begin{abstract}
In this short note, we study the behaviour of a product of matrices with a simultaneous renormalization. Namely, for any sequence $(A_n)_{n\in \N}$ of $d\times d$ complex matrices whose mean $A$ exists and whose norms' means are bounded, the product $\left(I_d + \frac1n A_0 \right) \dots\relax \left(I_d + \frac1n A_{n-1} \right) $ converges towards $\exp{A}$. We give a dynamical version of this result as well as an illustration with an example of "random walk" on horocycles of the hyperbolic disc.
\end{abstract}

\section{Introduction}
 
 Products of random matrices --- or cocycles --- are generally studied and well understood via ergodic theory, martingales on Markov chains, or spectral theory for instance. For some literature in this direction, one can look at a book such as \cite{benoist_quint}, the surveys \cite{furman,ledrappier_survey}. Indeed, results like the Osseledec theorem give a precise asymptotic behaviour of a product of random matrices. It provides informations like the logarithmic growth rate of the norm of the matrices. However, in \cite{tcl} we encountered a random product of matrices that did not fit the usually studied case. Indeed, understanding the limit of the characteristic functions of a renormalized sequence of probability measures had to be achieved by understanding, for any parameter $t$, a random product of matrices of the form $\left(A_{X_0} + \frac tn B_{X_0} \right) \dots\relax \left(A_{X_n} + \frac tn B_{X_n} \right) $ where $\left( X_n\right)_{n \in \N}$ is a binary sequences and $A_0, A_1, B_0, B_1$ are fixed $2\times 2$ matrices. The scale of normalization is different from the standard one, thus the result is more precise. We obtained a convergence of the matrices instead of convergence of the logarithm of the norms. Nevertheless, the method involved heavily relied on the properties of the matrices and, as such, was ad hoc for the problem we were interested in. However, in a an effort to replicate and generalize this type of random product of matrices (namely by understanding Corollary \ref{c.hyperbolic_walk}), we stumbled upon a surprising general property of these types of products that we explicit in Theorem \ref{t.general}.

\begin{theorem}\label{t.general}
Let $\left( A_n \right)_{n \in \N}$ be a sequence of $d\times d$ complex matrices satisfying
$$
\lim_{n\rightarrow +\infty}\frac1n \sum_{k=0}^{n-1}A_k=A.
$$
and such that  $\left( \frac1n \sum_{k=0}^{n}\|A_k\|\right)_{n\in \N^*}$ is bounded for a norm $\| \cdot \|$ by $\alpha$.

Define, for any $t$ in $\C$ and any positive integer $n$
$$
\Pi_n(t)=\left(I_d + \frac{t}{n} A_0 \right)\dots\relax \left(I_d + \frac{t}{n} A_n \right).
$$
Then, 
$$
\forall t \in \C, \ \lim_{n\rightarrow +\infty}\Pi_n(t)=\exp(tA).
$$
\end{theorem}

\begin{remark}
Here is an heuristic explanation of the statement of Theorem \ref{t.general}. Consider the  problem at the level of the Lie algebra.
The main term is $\frac tn \sum_{k=0}^{n-1}A_k$ and its limit $tA$. The limit of $\Pi_n$ is the exponential of the limit in the Lie algebra.
In a sense, at this scale, the behavior of the product is directed by the behavior of the sum $\sum_{k=0}^{n-1}A_k$ in the Lie algebra.
\end{remark}

An elementary version of Theorem \ref{t.general} for complex numbers is the following classical
\begin{lemma}\label{l.exp}
 Let $(u_n)_{n\in \N}$ be a bounded complex sequence whose mean converges towards $l$. Then
 $$
 \lim_{n\rightarrow +\infty}\prod_{k=0}^{n-1}\left( 1+ \frac{u_k}{n}\right)=e^l.
 $$
 \end{lemma}

\noindent{\bf Acknowledgement}
We wish to thank Jayadev Athreya and Samuel Leli\`evre for their interest in this problem and for sharing ideas with us.

 
 \section{Proof of Theorem \ref{t.general}}

\begin{proof}
First let us write, for any $n$ and $t$,
$$
\Pi_n(t)=\sum_{k=0}^{n-1}\left( \frac tn \right)^k\left(\sum_{0\leq i_1<...<i_k \leq n-1}A_{i_1}...A_{i_k} \right)
$$
and notice that for any $k$, 
$$
\left\|\left( \frac tn \right)^k\left(\sum_{0\leq i_1<...<i_k \leq n-1}A_{i_1}...A_{i_k} \right)\right\|\leq 
\left( \frac tn \right)^k \sum_{0\leq i_1<...<i_k \leq n-1}\left\|A_{i_1}\right\|...\left\|A_{i_k}\right\|.
$$
Notice that by commutativity
$$
\sum_{0\leq i_1<...<i_k \leq n-1}\left\|A_{i_1}\right\|...\left\|A_{i_k}\right\| <\frac{1}{k!}\left( \sum_{k=0}^n \|A_k\|\right)^k
$$
so we have
$$
\left\|\left( \frac tn \right)^k\left(\sum_{0\leq i_1<...<i_k \leq n-1}A_{i_1}...A_{i_k} \right)\right\|\leq\frac{|t^k \alpha^k|}{k!}.
$$
Hence, by the dominated convergence theorem, in order to prove the theorem, we only need to show that, for any $k$,
$$
\lim_{n\rightarrow +\infty}\frac 1{n^k}\sum_{0\leq i_1<...<i_k \leq n-1}A_{i_1}...A_{i_k} =\frac 1{k!}A^k.
$$
We proceed by induction on $k$. The case $k=1$ is the hypothesis of the theorem. Let us assume that this property is true for a fixed integer $k$.
$$
\frac 1{n^{k+1}}\sum_{0\leq i_1<...<i_k<l \leq n-1}A_{i_1}...A_{i_k}A_{l}
=
\frac 1{n^{k+1}}\sum_{l=k}^{n-1}\left(\sum_{0\leq i_1<...<i_k< \leq l-1}A_{i_1}...A_{i_k}\right)A_{l}
$$
and by induction hypothesis there is a sequence $(\epsilon_l)_{l\in\N}$ going to zero such that for any $l$:
$$
\left(\sum_{0\leq i_1<...<i_k< \leq l-1}A_{i_1}...A_{i_k}\right)=\frac{l^k}{k!}A^k + l^k\epsilon_l.
$$
Hence 
$$
\frac 1{n^{k+1}}\sum_{0\leq i_1<...<i_k<l \leq n-1}A_{i_1}...A_{i_k}A_{l}
=
\frac 1{n^{k+1}k!}\sum_{l=k}^{n-1}\left(l^k A^k + l^k\epsilon_l \right)A_{l}
$$
and it follows that 
$$
\frac 1{n^{k+1}}\sum_{0\leq i_1<...<i_k<l \leq n-1}A_{i_1}...A_{i_k}A_{l}
=
\frac 1{n^{k+1}k!}A^k\sum_{l=k}^{n-1}l^k A_{l} + \frac 1{n^{k+1}k!}\sum_{l=k}^{n-1}l^k\epsilon_lA_{l}
$$
So in order to get the result, we must prove that
$$
\lim_{n\rightarrow+\infty}\frac 1{n^{k+1}}\sum_{l=k}^{n-1}l^k A_{l}=\frac 1{k+1}A
$$
and 
$$
\lim_{n\rightarrow+\infty}\frac 1{n^{k+1}}\sum_{l=k}^{n-1}l^k \epsilon_l A_{l}=0.
$$
Since the sequence $(\epsilon_l)_{n\in \N}$ goes to $0$ as $l$ goes to $+\infty$, and since the mean of the norms of the $\|A_i\|$ are bounded by hypothesis, it is obvious that the first limit implies the second. 

Let us write
$$
\frac 1{n^{k+1}}\sum_{l=k}^{n-1}l^k A_{l}
=
\frac 1{n}\sum_{l=k}^{n-1}\left(\frac{l}{n}\right)^k A_{l}
$$
and state the following.

\begin{lemma}\label{l.pascal_ergodic}
Let $\left( u_n\right)_{n \in \N}$ be a sequence with values in $\mathcal{M}_d( \C)$ whose mean converges towards $L$ and let $g$ be a function in $\mathcal{C}^1(\R)$. Then,
 $$
\lim_{n\rightarrow+\infty}\frac{1}{n}\sum_{l=0}^{n-1}g\left( \frac{l}{n} \right)u_l=L\int_{0}^1g(t)dt.
 $$
\end{lemma}

Applying this lemma to  $(A_l)_{l\in \N}$ with $g:x\mapsto x^k$ yields the result.
\end{proof}

\begin{proof}[Proof of Lemma \ref{l.pascal_ergodic}]
We start with an Abel transform. With the notations of the lemma, let us denote, for any integer $n$, 
$$
S_n=\sum_{l=0}^n u_l
$$
and $S_{-1}=0$. For any $n$ in $\N$, 
$$
\frac{1}{n}\sum_{l=0}^{n-1}g\left( \frac{l}{n} \right)u_l=\frac{1}{n}\sum_{l=0}^{n-1}g\left( \frac{l}{n} \right)\left(S_l-S_{l-1}\right)
$$
and so
$$
\frac{1}{n}\sum_{l=0}^{n-1}g\left( \frac{l}{n} \right)u_l=\frac{1}{n}\sum_{l=0}^{n-1}g\left( \frac{l}{n} \right)S_l-\frac{1}{n}\sum_{l=0}^{n-1}g\left( \frac{l}{n} \right)S_{l-1}
$$
which yields
$$
\frac{1}{n}\sum_{l=0}^{n-1}g\left( \frac{l}{n} \right)u_l=\frac{1}{n}\sum_{l=0}^{n-2}S_l\left( g\left( \frac{l}{n} \right)-g\left( \frac{l+1}{n} \right)\right)+\frac1n S_{n-1}\cdot g\left(\frac{n-1}{n}\right).
$$
Now let us recall that $\lim_{n\rightarrow+\infty}\frac1n S_{n-1}=L$ hence there exists a sequence $(\epsilon_l)_{l\in \N}$ whose limit is zero such that for any $l$
$$
S_l=l(L+\epsilon_l).
$$
This, in turn, yields
$$
\frac{1}{n}\sum_{l=0}^{n-1}g\left( \frac{l}{n} \right)u_l
=
\frac{1}{n}\sum_{l=0}^{n-2}l(L+\epsilon_l)\left( g\left( \frac{l}{n} \right)-g\left( \frac{l+1}{n} \right)\right)
+
\frac1n (n-1)(L+\epsilon_{n-1})\cdot g\left(\frac{n-1}{n}\right).
$$
Now let us remark that
$$
\lim_{n\rightarrow +\infty}\frac1n (n\epsilon_{n-1}-L-\epsilon_{n-1})\cdot g\left(\frac{n-1}{n}\right)=0
$$
and that, since $g$ is differentiable, by the mean value theorem, for any couple of integers $l<n$, there exists a real $x$ in $[\frac ln, \frac{l+1}{n}]$ such that $g\left( \frac{l}{n} \right)-g\left( \frac{l+1}{n} \right)=\frac{g'(x)}{n}$. The function $g'$ being continuous, it is bounded on $[0,1]$, and since $\lim_{l\rightarrow +\infty}\epsilon_l=0$, we have
$$
\lim_{n\rightarrow+\infty}\frac{1}{n}\sum_{l=0}^{n-2}l\epsilon_l\left( g\left( \frac{l}{n} \right)-g\left( \frac{l+1}{n} \right)\right)=0.
$$
So, in order to complete the proof of this lemma, we must understand the asymptotic behaviour of
$$
\frac{L}{n}\sum_{l=0}^{n-2}l\left( g\left( \frac{l}{n} \right)-g\left( \frac{l+1}{n} \right)\right)
+
L\cdot g\left(\frac{n-1}{n}\right).
$$
First we write
$$
\frac{L}{n}\sum_{l=0}^{n-2}l\left( g\left( \frac{l}{n} \right)-g\left( \frac{l+1}{n} \right)\right)
+
L\cdot g\left(\frac{n-1}{n}\right)
=
\frac{L}{n}\sum_{l=1}^{n-2}g\left( \frac{l}{n} \right) - \frac{(n-2)}{n}L\cdot g\left(\frac{n-1}{n}\right)
+
L\cdot g\left(\frac{n-1}{n}\right).
$$
From here on, noticing that $\frac{1}{n}\sum_{l=1}^{n-2}g\left( \frac{l}{n} \right)$ is a Riemann sum yields the result, we indeed have
$$
\lim_{n\rightarrow+\infty}\frac{1}{n}\sum_{l=0}^{n-1}g\left( \frac{l}{n} \right)u_l=L\int_{0}^1g(t)dt.
$$
\end{proof}

\section{Applications to dynamics}



\begin{cor}
Let $\left( X,\mathcal{B},\mu,T\right)$ be a measured dynamical system, $\mu$ being an ergodic $T$-invariant probability measure. Let $A$ be function from $X$ to $\mathcal{M}_d(\C)$ such that each $A_{i,j}$ is in $L^1(X,\mu)$. Then, for almost every $x$ in $X$, 
$$
\lim_{n\rightarrow +\infty}\left( I_d +\frac{t}{n}A(x)\right) \dots\relax \left( I_d +\frac{t}{n}A\circ T^{n-1}(x)\right)=
\exp{\left(t  \int_{X} A(x)d\mu(x)\right)}.
$$
\end{cor}

\begin{proof}
This is just a matter of writing Theorem \ref{t.general} using Birkhoff's pointwise ergodic theorem.
\end{proof}


The fact that this theorem applies to generic points of ergodic probability measures  is useful to understand some dynamical systems as we illustrate with the following corollary which gives the asymptotic law of a "random walk" on horocycles of the hyperbolic disc.

\begin{cor}\label{c.hyperbolic_walk}
 Let $A_1=\begin{pmatrix}
           0 & 1 \\
           0 & 0
          \end{pmatrix}$
and $A_2=\begin{pmatrix}
           0 & 0 \\
           1 & 0
          \end{pmatrix}$.
Let $\mu$ be a shift-invariant ergodic probability measure on $\{1,2\}^\N$. Then, for $\mu$ almost every $x$ in $\{1,2\}^\N$, and every $t$ in $\R$,

$$
\left( I_2 +\frac{t}{n}A_{x_0}\right) \dots\relax \left( I_2 +\frac{t}{n}A_{x_{n-1}}\right)=\begin{pmatrix}
                                    \cosh\left( \frac{t}{\sqrt{\mu\left( [1] \right)\mu\left( [2] \right)}} \right)
                                    &
                                    \sqrt{\frac{\mu\left( [2] \right)}{\mu\left( [1] \right)}}\sinh\left( \frac{t}{\sqrt{\mu\left( [1] \right)\mu\left( [2] \right)}} \right)
                                    \\
                                    \sqrt{\frac{\mu\left( [0] \right)}{\mu\left( [2] \right)}}\sinh\left( \frac{t}{\sqrt{\mu\left( [1] \right)\mu\left( [2] \right)}} \right)
                                    &
                                    \cosh\left( \frac{t}{\sqrt{\mu\left( [1] \right)\mu\left( [2] \right)}} \right)
                                   \end{pmatrix}.
$$
\end{cor}
\begin{proof}
One easily computes $\exp \begin{pmatrix}
							0          &  t \mu([1]) \\
							t \mu([2]) &  0
					      \end{pmatrix}$ to get the desired result.

\end{proof}

 \begin{figure}[!ht]
  \center\includegraphics[scale=0.5]{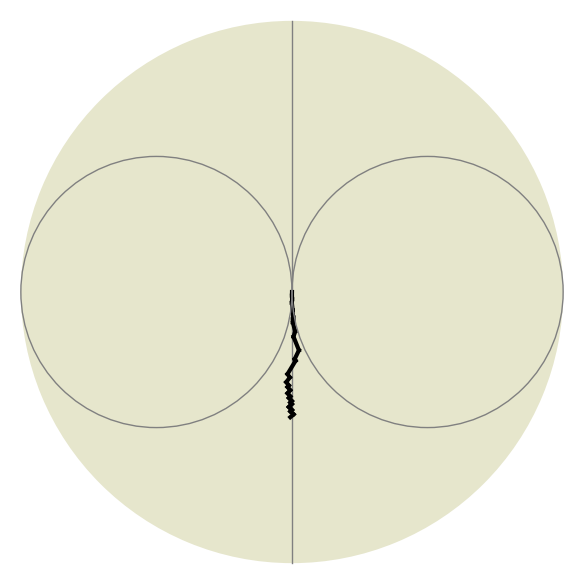}
  \caption{An illustration of a random hyperbolic walk}\label{f.hyp_disc}
 \end{figure}

 \begin{remark}
 Remark that Corollary \ref{c.hyperbolic_walk} can be interpreted in a geometric way given that $I_2+A_1$ and $I_2+A_2$ are generators of $SL_2(\Z)$. Notice that taking, for instance, $\mu$ to be the symmetric Bernoulli measure on $\{1,2\}^\N$, one gets $  \Pi_X(t)=\begin{pmatrix}
                                    \cosh\left( \frac{t}{2} \right)
                                    &
                                    \sinh\left( \frac{t}{2} \right)
                                    \\
                                    \sinh\left( \frac{t}{2} \right)
                                    &
                                    \cosh\left( \frac{t}{2} \right)
                                   \end{pmatrix}
                                   $
which is conjugated to $\begin{pmatrix}
                                    e^{\frac t2}
                                    &
                                    0
                                    \\
                                    0
                                    &
                                    e^{-\frac t2}
                                    \end{pmatrix}$
by a rotation of angle $\pi/4$. Hence almost every "random walk" (with simultaneous renormalization) on two horocycles of the hyperbolic disc converges towards a unique point on the geodesic represented by the vertical diameter of the disc as illustrated on Figure \ref{f.hyp_disc}\footnote{The calculations are made in the hyperbolic plane but the picture is presented in the disc since it is more symmetric.}.

 \end{remark}

\bibliographystyle{plain}
\bibliography{biblio}

\end{document}